\documentclass[english, 11pt]{amsart}
\usepackage{amssymb}
\usepackage{amsthm}
\usepackage{amsfonts}
\usepackage{amscd}
\usepackage{mathrsfs}
\usepackage{blindtext}
\usepackage{geometry}
\usepackage[T1]{fontenc}
\usepackage{enumerate}%
\usepackage{amsthm,lipsum}
\usepackage{tikz}
\usepackage[colorlinks, linkcolor=red]{hyperref}
\usepackage{color}
\usepackage[utf8]{inputenc}
\usepackage{dsfont}
\definecolor{immi}{rgb}{0,.6,.1}

\newbox\removebox
\newcommand\remove[2]{%
\setbox\removebox=\ifmmode\hbox{$#2$}\else\hbox{#2}\fi%
\leavevmode
\rlap{\textcolor{#1}{\vrule height0.8ex depth-0.6ex width\wd\removebox}}%
\box\removebox
}
\long\def\bigremove#1{%
\par\setbox\removebox=\vbox{#1}%
\vbox{%
\vbox to0pt{\hbox{\tikz\draw[color=blue,thick] (0,0) -- (\wd\removebox,-\ht\removebox)  (\wd\removebox,0) -- (0,-\ht\removebox);}}
\box\removebox
}
}

\usepackage{mathrsfs} %for the \mathscr
\usepackage{bbm}

\makeatletter
\def\RFss@@#1{\RF^*_{\!*#1}}
\def\RFss@_#1{\RFss@@{,#1}}
\def\RFss{\@ifnextchar_{\RFss@}{\RFss@@{}}}
\makeatother

\newcommand{\RF}{{\rm RF}}

\def\11{{\mathbf 1}}

\newtheorem{thm}[subsection]{Theorem}
\newtheorem{lem}[subsection]
{Lemma}
{Corollary}

{Conjecture}
{Problem}
\theoremstyle{plain}
\newtheorem*{namedthm}{\namedthmname}
\newcounter{namedthm}


\theoremstyle{definition}

\theoremstyle{remark}
{Remark}
\newtheorem{rem}[subsection]
{Remark}
{Remarks}

\theoremstyle{plain}

  {\par\medskip\noindent #1\par\begingroup%
    \advance\leftskip by 1em\advance\rightskip by 1em}%
  {\par\endgroup}

\renewcommand{\phi}{\varphi}
\renewcommand{\epsilon}{\varepsilon}
\renewcommand{\theta}{\vartheta}

\renewcommand{\and}{ \quad \text{and} \quad }

\begin{document}

\setcounter{tocdepth}{1} % Show subsection in table of contents

\author[T.~Thu~Nguyen]
{Thi Thu Nguyen}
 \address{Université de Lille, Laboratoire Paul Painlevé, 
59655 Villeneuve d'Ascq Cedex, France}
\email{thuclcsphn@gmail.com}
\email{thithu.nguyen@univ-lille.fr}

\begin{abstract} We study an asymptotic formula for average orders of Goldbach representations of an integer as the sum of $k$ primes. We extend the existing result for $k=2$ to a general $k$, for which we obtain a better error term. Moreover, we prove an equivalence between the Riemann Hypothesis and a good average order in this case.

\end{abstract}

\title[Sums of several primes]{Goldbach Representations with several primes} 

\maketitle
%\tableofcontents

\section{Introduction}\label{sec:intro}

%\section{}
Let $k\geq 2$ be an integer, we define the weighted Goldbach function in the general case
$$G_k(n) =\sum_{n_1+\dots+n_k=n}\Lambda(n_1)\dots\Lambda(n_k),$$
where $\Lambda(n)$ is the von Mangoldt function. For $k\geq 2$, the expected asymptotic formula for $G_k(n)$ is of the form
\begin{align*}
    G_k(n)=\dfrac{n^{k-1}}{(k-1)!} \mathfrak{S}_k(n)+ \text{error term},
\end{align*}
where 
$$\mathfrak{S}_k(n):= \prod_{p\mid n}\left( 1-\left(\dfrac{-1}{p-1}\right)^{k-1}\right)\prod_{p\nmid n}\left( 1-\left(\dfrac{-1}{p-1}\right)^{k}\right),$$ (see \cite{FG}).
The first result of this type was studied by Hardy and Littlewood \cite{HL} which inspired Vinagradov \cite{Vi} for $k=3$. Friedlander and Goldston \cite{FG} established this bound for each $k \geq 5$ and got a slightly weaker estimate for $k = 3$ and $4$.
 Moreover, Friedlander and Goldston (Corollary, \cite{FG}) proved that for $k\geq 5$ the Generalized Riemann Hypothesis is equivalent to the estimate
\begin{align*}
     G_k(n)=\dfrac{n^{k-1}}{(k-1)!} \mathfrak{S}_k(n)+ \mathcal{O}(n^{k-3/2}).
\end{align*} 
In this paper, we will consider its average order, denoted by 
$$S_k(X):= \sum_{n\leq X} G_k(n).$$ Studying such average orders is a standard practice in analytic number theory. Here a special motivation is that according to Granville \cite{granville1, granville2}, the average order of $G_k(n)$ can be related to the Riemann Hypothesis (RH). Languasco and Zaccagnini \cite{Lang-Zacca1} proved the following asymptotic result  for $S_2(X)$, under the RH.
\begin{thm}\label{0}
    Assuming the RH and let $X \geq 2$. Then 
    $$S_2(X)=\dfrac{X^2}{2}-2\sum_\rho\dfrac{X^{\rho+1}}{\rho(\rho+1)}+\mathcal{O}(X\log^3X).$$
\end{thm}

We expect a similar formula in the general case  $k \geq 2$. Languasco and Zaccagnini \cite{Lang-Zacca1} stated an asymptotic formula of $S_k(X)$ for $k\geq 3$, with the error term  $\mathcal{O}_k(X^{k-1}\log^kX)$ but did not prove it in detail. In fact, we can prove this result by using the original Hardy and Littlewood circle method with the infinite exponential sum.
Note that Theorem \ref{0} was later proved by another method \cite{GY}, which uses the finite exponential sum as studied by Bhowmik and Schlage-Puchta \cite{BP}. Here we improve the asymptotic formula of $S_k(X)$ in \cite{Lang-Zacca1} with the method of Goldston and Yang \cite{GY}, who showed a good estimate of the expected value function (see Lemma \ref{EX}). While Languasco and Zaccagnini used Lemma \cite{LZ} for their bound, we use a trivial bound (see Lemma \ref{S_0}).   Surprisingly we get a much better estimate, the $\log$-power is always $3$ for all $k \geq 2$.

\begin{thm}\label{1}
    Let $k\geq 2$, $X\geq k$ and assume the RH holds. Then we have 
    $$S_k(X)=\dfrac{X^k}{k!}+H_k(X)+\mathcal{O}_k(X^{k-1}\log^3X),$$
 with $$H_k(X)=-k\sum_\rho\dfrac{X^{\rho+k-1}}{\rho(\rho+1)...(\rho+k-1)},$$
 where $\rho$ are the non-trivial zeros of  Riemann zeta function with $Re(\rho)=1/2$.
\end{thm}

We note that the error term in Theorem \ref{0} is essentially the best possible because we know the error term is $\Omega(X\log \log X)$\cite{BP}.
Similar to the case $k=2$, we expect an omega-result of the average order in the general case as was studied by Bhowmik, Schlage-Puchta \cite{BP-omega}, who proved the error term is $\Omega(X^{k-1})$, while \cite{BRP} shows a similar result for the error term to be $\Omega_\pm(X^{k-1})$. In this paper,  we prove the following result.
\begin{thm}\label{omega}
    Let $k\geq 2$, we have 
    $$S_k(X)=\dfrac{X^k}{k!}+H_k(X)+ \Omega(X^{k-1}\log \log X).$$
\end{thm}
To do that, we use the idea in \cite{BP} for $k=2$. We show that for $n$ sufficiently large, $G_k(n)=\Omega(n^{k-1}\log \log n$). Then Theorem \ref{omega} will be proved because if Theorem \ref{omega} is false, that means
\begin{align*}
     S_k(n)=\dfrac{n^k}{k!}+H_k(n)+ o(n^{k-1}\log \log n).
\end{align*}
This implies
\begin{align*}
    G_k(n)=S_k(n)-S_k(n-1)=o(n^{k-1}\log \log n).
\end{align*}

For $k=2$, we know that there is a good relation between the Riemann Hypothesis and the average order, the RH is equivalent to the estimation $$S_2(X)=\dfrac{X^2}{2}+\mathcal{O}(X^{3/2+\epsilon}),$$ for any $\epsilon >0$, as mentioned in \cite{BHMS} and \cite{Bhowmik-Ruzsa}. The method of \cite{Bhowmik-Ruzsa} was generalized in \cite{BCSS} to obtain a zero-free region for the Riemann zeta-function.

In this paper, we prove that for $k\geq 2$, a good estimation of $S_k(X)$ is equivalent to the Riemann Hypothesis. 
\begin{thm}\label{2}
    Let $k\geq 2$ and $X\geq k$, the RH is equivalent to
    $$S_k(X)=\dfrac{X^k}{k!}+\mathcal{O}_k(X^{k-1/2+\epsilon}),$$ for any $\epsilon >0$.
\end{thm}
In fact, we show that Theorem \ref{2} is a consequence of the following theorem. 
\begin{thm}[Quasi-Riemann Hypothesis]\label{3}
    Let $k\geq 2$. We assume that there exists $0<\delta<1$ such that $$S_k(X)=\dfrac{X^k}{k!}+\mathcal{O}_k(X^{k-\delta}),$$ then for any non-trivial zero $\rho$ of Riemann zeta function, we have $\Re (\rho) < 1$.
\end{thm}

\section{Proof of Theorem \ref{1}}
To prove Theorem \ref{1}, we use induction on $k$. This theorem is true when $k=2$ (Theorem \ref{0}) and suppose that it holds up to $k-1$, then we prove it for $k$. We use the notation of \cite{GY}.  Consider the generating function
\begin{align*}
    S_0(\alpha,x)=\sum_{n \leq x}\Lambda_0(n)e(n\alpha), \hspace{1cm} e(\alpha)=e^{2\pi i \alpha},
\end{align*}
where $\Lambda_0(n)=\Lambda(n)-1.$ Then for $k\geq 2$, we have
\begin{align*}
     S_0(\alpha,x)^k&=\sum_{n_1,\dots, n_k \leq x}\Lambda_0(n_1)\dots\Lambda_0(n_k)e((n_1+\dots+n_k)\alpha)\\
     &=\sum_{n\leq kx}B_k(n, x)e(n\alpha),
\end{align*}
where 
\begin{align*}
    B_k(n, x)=\sum_{\substack{n_1,\dots, n_k \leq x\\n_1+\dots+n_k=n}}\Lambda_0(n_1)\dots\Lambda_0(n_k).
\end{align*}
When $k\leq n\leq x$, we can express $G_k(n)$ through $B_k(n,x)$ as
\begin{align}\label{B_k}
    \begin{split}
        B_k(n, x)&=\sum_{n_1+\dots+n_k=n}\Lambda_0(n_1)\dots\Lambda_0(n_k)\\
        &=\sum_{n_1+\dots+n_k=n}(\Lambda(n_1)-1)\dots(\Lambda(n_k)-1)\\
        &=G_k(n)-k\sum_{n_1=1}^{n-k+1}G_{k-1}(n-n_1)+\sum_{i=2}^{k-2}(-1)^i \binom{k}{i}\sum_{n_1+\dots+n_k=n}\Lambda(n_i)\dots\Lambda(n_k)\\
        &+(-1)^{k-1}k\sum_{n_1+\dots+n_k=n}\Lambda(n_k)+(-1)^k\sum_{n_1+\dots+n_k=n}1.
    \end{split}
\end{align}
Let $$I(X,\alpha)=\sum_{n\leq X}e(n\alpha).$$ We then have the estimate  $I(X,\alpha)\ll \min \left(X, \dfrac{1}{||\alpha||}\right).$ For $x\geq X$, We have
\begin{align*}
    \int_0^1 S_0(\alpha,x)^kI(X,-\alpha)d\alpha= \sum_{n\leq X}B_k(n,x).
\end{align*}
Substituting \eqref{B_k} into the above, we obtain
\begin{align}\label{tong S_k}
\begin{split}
     S_k(X)&=\int_0^1 S_0(\alpha,x)^kI(X,-\alpha)d\alpha+\sum_{i=1}^{k-2}(-1)^{i+1} \binom{k}{i}\sum_{n\leq X}\sum_{n_1+\dots+n_k=n}\Lambda(n_{i+1})\dots\Lambda(n_k)\\
     &+(-1)^{k}k\sum_{n\leq X}\sum_{n_1+\dots+n_k=n}\Lambda(n_k)+(-1)^{k+1}\sum_{n\leq X}\sum_{n_1+\dots+n_k=n}1\\
        &=:I_0+ \sum_{i=1}^{k-2}(-1)^{i+1} \binom{k}{i}I_i+(-1)^{k}kI_{k-1}+(-1)^{k+1}I_k.
\end{split}
\end{align}
We estimate $I_i$, for $0\leq i\leq k$.
\subsection{Main term of $I_k$.}
We have
\begin{align}\label{I_k}
    I_k=\sum_{n\leq X}\sum_{n_1+\dots+n_k=n}1=\sum_{k\leq n\leq X}\binom{n-1}{k-1}= \dfrac{X^k}{k!}+\mathcal{O}(X^{k-1}).
\end{align}

\subsection{Main term of $I_{k-1}$.}
\begin{align*}
    I_{k-1}&=\sum_{n\leq X}\sum_{n_1+\dots+n_k=n}\Lambda(n_k)=\sum_{n\leq X} \sum_{n_k=1}^{X-(k-1)}\sum_{n_1+\dots+n_{k-1}=n-n_k}\Lambda(n_k)\\
    &=\sum_{n\leq X} \sum_{n_k=1}^{X-(k-1)}\binom{n-n_k-1}{k-2}\Lambda(n_k)=\sum_{n_k=1}^{X-(k-1)}\Lambda(n_k) \left(\dfrac{(X-n_k)^{k-1}}{(k-1)!}+\mathcal{O}(X^{k-2})\right)\\
    &=\psi_{k-1}(X-k+1)+\mathcal{O}(X^{k-1}),
\end{align*}
where for a non-negative integer $j$, $$\psi_j(x):=\dfrac{1}{j!}\sum_{n\leq x}\Lambda(n)(x-n)^j.$$When $j=0$, we have $\psi_0(x)=\psi(x)$.
We have some properties of this function.

 Firstly, 
 \begin{align}\label{j+1}
       \psi_j(x+1)=\psi_j(x)+\mathcal{O}(x^j).
 \end{align}
We note that, 
\begin{align*}
    \psi_{j}(x) = \int_0^x \psi_{j-1}(t)dt.
\end{align*}
Moreover, for $j=1$ we have an explicit formula (13.7 \cite{MV})
\begin{align*}
    \psi_1(x)&=\dfrac{x^2}{2}-\sum_\rho \dfrac{x^{\rho+1}}{\rho(\rho+1)}-\dfrac{\zeta'}{\zeta}(0)x+\dfrac{\zeta'}{\zeta}(-1)+\mathcal{O}(x^{-1/2}) \\
    &=\dfrac{x^2}{2}-\sum_\rho \dfrac{x^{\rho+1}}{\rho(\rho+1)} +\mathcal{O}(x).
\end{align*}
Then by induction, we obtain 
\begin{align}\label{j}
     \psi_j(x)= \dfrac{x^{j+1}}{(j+1)!}-\sum_\rho \dfrac{x^{\rho+j}}{\rho(\rho+1)...(\rho+j)}+\mathcal{O}(x^{j}).
\end{align}

From \eqref{j+1}, \eqref{j} we obtain
\begin{align}\label{I_k-1}
    I_{k-1}=\dfrac{X^k}{k!}-\sum_\rho \dfrac{X^{\rho+k-1}}{\rho(\rho+1)...(\rho+k-1)}+\mathcal{O}(X^{k-1}).
\end{align}
\subsection{Estimate $I_0$.}
Similar to the idea of Goldston and Yang \cite{GY}, we define an expected value function by
 $$E_X(S_0(\alpha)):=\dfrac{1}{X}\int_X^{2X}S_0(\alpha,x)dx. $$
We need the following lemmas.
\begin{lem}[Lemma 7, \cite{GY}]\label{EX} Assuming the RH, we have for $1 \leq h \leq X$
$$\int_{-1/2h}^{1/2h}E_X(\left|S_0(\alpha)\right|^2)d\alpha \ll \dfrac{X\log ^2X}{h}.$$
\end{lem}
\begin{lem}\label{S_0}
    Let $x\in [X, 2X]$, we have the estimate
    $$S_0(\alpha, x) \ll X.$$
\end{lem}
\begin{proof}
    We have 
    \begin{align*}
        S_0(\alpha, x)&\ll \sum_{n\leq x} \left|\Lambda(n)-1\right||e(n\alpha)| \leq \sum_{n\leq x} \left|\Lambda(n)-1\right|\\
        &\leq  \sum_{n\leq x}\Lambda(n)+\sum_{n\leq x}1=\psi(x)+[x]\\
        &\ll X,
    \end{align*}
    with $x \leq 2X$.
\end{proof}
Consider $\alpha \in [-1/2, 1/2]$, since $I(X,-\alpha) \ll \min\left(X, \dfrac{1}{||\alpha||}\right)$, we have
\begin{align}\label{I_0'}
\begin{split}
    I_0&\ll \int_{-1/2}^{1/2}\left|S_0(\alpha,x)\right|^kI(X,-\alpha)d\alpha= \int_{-1/2}^{1/2}E_X(\left|S_0(\alpha)\right|^k)I(X,-\alpha)d\alpha\\
& \ll X\int_{-1/X}^{1/X}E_X(\left|S_0(\alpha)\right|^k)d\alpha+ \int_{1/X}^{1/2}\dfrac{E_X(\left|S_0(\alpha)\right|^k)}{\alpha}d\alpha.
\end{split}
\end{align}
We estimate the first term of \eqref{I_0'}
\begin{align*}
&X\int_{-1/X}^{1/X}E_X(\left|S_0(\alpha)\right|^k)d\alpha = X\int_{-1/X}^{1/X} \dfrac{1}{X}\int_X^{2X}\left|S_0(\alpha,x)\right|^kdx d\alpha\\
&\leq X\int_{-1/X}^{1/X} \dfrac{1}{X} \max_{x\in [X,2X]}\left|S_0(\alpha,x)\right|^{k-2}\int_X^{2X}\left|S_0(\alpha,x)\right|^2dx d\alpha\\
&\leq X \max_{\substack{x\in [X,2X]\\ |\alpha| \leq 1/X}}\left|S_0(\alpha,x)\right|^{k-2}\int_{-1/X}^{1/X}E_X(\left|S_0(\alpha)\right|^2)d\alpha \\
&\ll X^{k-1}\log^2X, 
\end{align*}
where for the last inequality, we use  Lemma \ref{EX} and Lemma \ref{S_0}.

For the second term of \eqref{I_0'}, writing  $[1/X, 1/2]$ as the disjoint union of $[2^{j}/X, 2^{j+1}/X]$ for $0\leq j \leq \mathcal{O}(\log X)$, then 
\begin{align*}
\int_{1/X}^{1/2}\dfrac{E_X(\left|S_0(\alpha)\right|^k)}{\alpha}d\alpha &\ll \sum_{j=0}^{\mathcal{O}(\log X)}\dfrac{X}{2^j} \int_{2^j/X}^{2^{j+1}/X}E_X(\left|S_0(\alpha)\right|^k)d\alpha\\
&\ll   \sum_{j=0}^{\mathcal{O}(\log X)}\dfrac{X}{2^j} \int_{2^j/X}^{2^{j+1}/X} \dfrac{1}{X} \max_{x\in [X,2X]}\left|S_0(\alpha,x)\right|^{k-2}\int_X^{2X}\left|S_0(\alpha,x)\right|^2dx d\alpha\\
&\leq  \sum_{j=0}^{\mathcal{O}(\log X)} \dfrac{X}{2^j}  \max_{\substack{x\in [X,2X]\\ \alpha \in [2^j/X, 2^{j+1}/N]}}\left|S_0(\alpha,x)\right|^{k-2}\int_{2^j/X}^{2^{j+1}/X} E_X(\left|S_0(\alpha)\right|^2)d\alpha\\
 &\ll \sum_{j=0}^{\mathcal{O}(\log X)}\dfrac{X}{2^j}X^{k-2} X\dfrac{2^{j+1}}{X}\log^2X \ll X^{k-1}\log^3X.
\end{align*}
Then we obtain
\begin{align}\label{I_0}
    I_0 \ll X^{k-1}\log^3X.
\end{align}
\subsection{Main term of $I_i$, $1\leq i\leq k-2$.}
\begin{align*}
    I_i&=\sum_{n\leq X} \sum_{n_1+\dots+n_k=n}\Lambda(n_{i+1})\dots\Lambda(n_k)\\
    &=\sum_{n\leq X}\sum_{n_1+\dots+n_i=i}^{n-(k-i)}\left(\sum_{n_{i+1}+\dots+n_k=n-(n_1+\dots+n_i)}\Lambda(n_{i+1})\dots\Lambda(n_k)\right)\\
    &=\sum_{n\leq X}\left[\binom{i-1}{i-1}G_{k-i}(n-i)+ \binom{i}{i-1}G_{k-i}(n-1-i)+\dots+\binom{n-(k-i+1)}{i-1}G_{k-i}(k-i)\right]\\
    &=\binom{i-1}{i-1}G_{k-i}(X-i)+\left[\binom{i-1}{i-1}+\binom{i}{i-1}\right]G_{k-i}(X-1-i)+\dots\\
    &+\left[\binom{i-1}{i-1}+\binom{i}{i-1}+\dots+\binom{i+X-k-1}{i-1}\right]G_{k-i}(k-i).
\end{align*}
Using the formula for $m$ non-negative integer
\begin{align}
    \binom{i-1}{i-1}+\binom{i}{i-1}+\dots+\binom{i+m}{i-1}=\binom{i+m+1}{i},
\end{align}
we rewrite 
\begin{align*}
    \begin{split}
        I_i&=\binom{i}{i}G_{k-i}(X-i)+\binom{i+1}{i}G_{k-i}(X-i)+\dots+ \binom{i+X-k}{i}G_{k-i}(k-i)\\
        &=\sum_{n\leq X}\binom{X-n}{i}G_{k-i}(n)\\
        &=\sum_{n\leq X}\dfrac{(X-n)^i}{i!}G_{k-i}(n)+\mathcal{O}\left(\sum_{n\leq X}(X-n)^{i-1}G_{k-i}(n)\right)\\
        &=T_i(X,k-i)+\mathcal{O}(X^{k-1}),
    \end{split}
\end{align*}
where for $j \geq 0$, 
$$T_j(X,k-i):=\dfrac{1}{j!}\sum_{n\leq X}(X-n)^jG_{k-i}(n).$$
Then we have a property for this function, that is 
\begin{align*}
    T_{j+1}(X,k-i)=\int_0^X T_j(t,k-i) dt
\end{align*}
Moreover, by the induction hypothesis, for $1\leq i\leq k-2$
\begin{align*}
    \sum_{n\leq x}G_{k-i}(n)=\dfrac{x^{k-i}}{(k-i)!}+H_{k-i}(x)+\mathcal{O}_k(x^{k-i-1}\log^3x).
\end{align*}
So we calculate
\begin{align}\label{Tj}
    T_j(X,k-i)= \dfrac{X^{k-i+j}}{(k-i+j)!}-(k-i)\sum_\rho\dfrac{X^{\rho+k-i+j-1}}{\rho(\rho+1)...(\rho+k-i+j-1)}+\mathcal{O}_k(X^{k-i+j-1}\log^3 X).
\end{align}
Replacing $j=i$ in \eqref{Tj},   we obtain
\begin{align}\label{I_i}
    I_i=\dfrac{X^k}{k!}-(k-i)\sum_\rho\dfrac{X^{\rho+k-1}}{\rho(\rho+1)...(\rho+k-1)}+\mathcal{O}_k(X^{k-1}\log^3 X).
\end{align}
Combining \eqref{tong S_k}, \eqref{I_k}, \eqref{I_k-1}, \eqref{I_0} and \eqref{I_i}, we obtain
\begin{align*}
    S_k(X)&=\dfrac{X^k}{k!}\left[\sum_{i=1}^k(-1)^{i+1}\binom{k}{i}\right]-\sum_\rho\dfrac{X^{\rho+k-1}}{\rho(\rho+1)...(\rho+k-1)}\left[\sum_{i=1}^k(-1)^{i+1}\binom{k}{i}(k-i)\right]\\
    &+\mathcal{O}(X^{k-1}\log^{3} X)\\
    &=\dfrac{X^k}{k!}-k\sum_\rho\dfrac{X^{\rho+k-1}}{\rho(\rho+1)...(\rho+k-1)}+\mathcal{O}_k(X^{k-1}\log^3 X),
\end{align*}
where for the last equation, we use
\begin{align*}
    \sum_{i=0}^k(-1)^{i}\binom{k}{i}&=(1-1)^k=0\\
    \sum_{i=0}^k(-1)^{i}\binom{k}{i}(k-i)&=k(1-1)^{k-1}=0.
\end{align*}
Thus the proof is complete.
\section{Proof of the Quasi-Riemann Hypothesis}
In this part, we consider the power series for $|z|<1$,
\begin{align}
    F(z)= \sum_{n\geq 1}\Lambda(n)z^n
\end{align}
We take the $k^{th}$ power of $F(z)$ and obtain
\begin{align*}
    F(z)^k= \sum_{n\geq 1}G_k(n)z^n=(1-z) \sum_{n\geq 1}S_k(n)z^n
\end{align*}
By the assumption in Theorem \ref{3}, we have
\begin{align}\label{s}
    \begin{split}
         \sum_{n\geq 1}S_k(n)z^n&=\sum_{n\geq 1}\left(\dfrac{n^k}{k!}+\mathcal{O}_k(n^{k-\delta})\right)z^n\\
         &=\dfrac{1}{k!}\sum_{n\geq 1}n^kz^n+ \mathcal{O}_k\left(\sum_{n\geq 1}n^{k-\delta}|z|^n\right).
    \end{split}
\end{align}
We then evaluate the main term in \eqref{s} as follows.
\begin{lem}\label{lem1}
    For any integer $k\geq 2,$ we have
    \begin{align*}
        \sum_{n \geq 1}n^kz^n =\dfrac{k!}{(1-z)^{k+1}}+ \mathcal{O}(|1-z|^{-k}).
    \end{align*}
\end{lem}
\begin{proof}
We note that there exists unique real numbers $a_0, a_1,\dots, a_{k}$ such that 
   \begin{align}\label{n^kz^n}
       \sum_{j=0}^k \binom{n+j}{j}a_j=n^k
   \end{align}
   holds for integers $n=0, 1, \dots, k$. Then \eqref{n^kz^n} holds for all positive integers $n$.
From \eqref{n^kz^n}, we can rewrite
\begin{align*}
    \sum_{n \geq 1}n^kz^n = \dfrac{a_k}{(1-z)^{k+1}}+ \dfrac{a_{k-1}}{(1-z)^{k}}+ \dots +\dfrac{a_0}{1-z}.
\end{align*}
We calculate the value of $a_k$ 
\begin{align}
    a_k=k^k -\binom{k}{1}(k-1)^k+\binom{k}{2}(k-2)^k-\dots+(-1)^{k-1}\binom{k}{k-1}.
\end{align}
 We just need to prove
   \begin{align*}
       a_k=k!.
   \end{align*}
Let $k \geq 2$, we define the function
\begin{align}
    f_{k,i}=k^i -\binom{k}{1}(k-1)^i+\binom{k}{2}(k-2)^i-\dots+(-1)^{k-1}\binom{k}{k-1}.
\end{align}
We note that 
\begin{align}\label{ki}
     f_{k,i}=k(f_{k,i-1}+f_{k-1,i-1})
\end{align}
and 
\begin{align}\label{f0}
     f_{k,i}=0, \text{ for all }1\leq i \leq k-1.
\end{align}
In fact, we prove \eqref{f0} by induction on $i$.
When $i=1$, 
\begin{align*}
    f_{k,1}&=k -\binom{k}{1}(k-1)+\binom{k}{2}(k-2)-\dots+(-1)^{k-1}\binom{k}{k-1}\\
    &=\binom{k}{1}-2\binom{k}{2}+3\binom{k}{3}-\dots+(-1)^{k-1}k\binom{k}{k},
\end{align*}
since $$(k-n)\binom{k}{n}=(n+1)\binom{k}{n+1}, \text{ for all } 0\leq n \leq k.$$
On the other hand, 
\begin{align*}
    (1-x)^k=\sum_{j=0}^k(-1)^j\binom{k}{j}x^j,
\end{align*}
and its derivative
\begin{align}\label{fk}
    k(1-x)^{k-1}=\sum_{j=0}^k(-1)^jj\binom{k}{j}x^{j-1}
\end{align}
Replace $x$ in \eqref{fk} by $1$, we obtain $f_{k,1}=0$.

Next, assuming that \eqref{f0} is true for $1\leq i \leq k-2$, we now prove that it is true for $i+1$. By \eqref{ki}, we have
\begin{align*}
    f_{k,i+1}=k(f_{k,i}+f_{k-1,i})=0.
\end{align*}
From \eqref{ki} and \eqref{f0}, we obtain 
\begin{align*}
    f_{k,k}=kf_{k-1,k-1}=\dots=k!f_{1,1}=k!.
\end{align*}
Thus $$a_k=f_{k,k}=k!.$$ 
\end{proof}
Using Lemma \ref{lem1}, we rewrite \eqref{s} as follows
\begin{align*}
    \sum_{n\geq 1}S_k(n)z^n=\dfrac{1}{(1-z)^{k+1}}+\mathcal{O}_k\left((1-|z|)^{\delta-k-1}\right).
\end{align*}
Hence, we obtain 
\begin{align}\label{F}
\begin{split}
    F(z)^k&= (1-z)\sum_{n\geq 1}S_k(n)z^n\\
    &=\dfrac{1}{(1-z)^{k}}+\mathcal{O}_k\left(|1-z|(1-|z|)^{\delta-k-1}\right)\\
    &=\dfrac{1}{(1-z)^{k}}+\mathcal{O}_k\left(|1-z|N^{k+1-\delta}\right)
\end{split}
\end{align}
on the circle $|z|=R=1-\frac{1}{N}$, for a large positive integer $N$.
We remark that the error term is less than the absolute value of the main tern if $|1-z|<N^{\frac{\delta}{k+1}-1}.$ We call this is a major arc on $|z|=R$, denoted by  $\mathfrak{M}$. The rest of the circle is called a minor arc, denoted by $\mathcal{M}$. 

We consider the major arc. Taking the complex $k^{th}$ root of \eqref{F}, then $F(z)$ can be written as 
\begin{align*}
    F(z)= \dfrac{\omega}{1-z}+ \mathcal{O}_k\left(|1-z|^kN^{k+1-\delta}\right),
\end{align*}
where $\omega$ is a $k^{th}$ root of unity. Then we prove $\omega=1$. In fact, on the circle $|z|=R$, we choose $z$ a real number which tends to $1$, for example $z=e^{-1/N}$. Putting $c:=\frac{1}{1-z}$ a real number, then $c$ tends to infinity and $|1-z|^kN^{k+1-\delta}$ tends to $0$. Assuming $$\omega=a+bi, \text{ for }a,b \in \mathbb{R}, b\neq 0,$$
then, 
$$F(z)=ac+bci.$$
This is impossible because from the definition of $F(z)$, $F(z)$ is real if $z$ is a real number. Hence $\omega$ is also a real number. We consider two cases for the integer $k$.\\
Case 1: if $k$ is odd, this implies $\omega=1$.\\
Case 2: if $k$ is even, $\omega=\pm 1.$ Since $F(z)$ is continuous and its coefficients are non-negative, the sign of $\omega$ is $+$, i.e, $\omega=1$. 

We conclude 
\begin{align}\label{F(z)}
     F(z)= \dfrac{1}{1-z}+ \mathcal{O}_k\left(|1-z|^kN^{k+1-\delta}\right) \text{ on } \mathfrak{M}.
\end{align}

Now we introduce the kernel function,
\begin{align*}
    K(z)=z^{-N-1}\dfrac{1-z^N}{1-z},
\end{align*} then $K(z) \ll |1-z|^{-1}$. By Cauchy's integral formula, the Chebyshev function can be written as
\begin{align}
\begin{split}\label{psi}
     \psi(N)&=\dfrac{1}{2\pi i} \int_{|z|=R}F(z)K(z)dz\\
    &=N+ \dfrac{1}{2\pi i} \int_{|z|=R}\left(F(z)-\dfrac{1}{1-z}\right)K(z)dz.
\end{split}
\end{align}
We split the circle $|z|=R$ into the major arc $\mathfrak{M}$ and the minor arc $\mathcal{M}$. 

On $\mathfrak{M}$, from \eqref{F(z)}, we obtain 
\begin{align*}
    \int_{\mathfrak{M}}\left(F(z)-\dfrac{1}{1-z}\right)K(z)dz&\ll_k \int_{\mathfrak{M}} \left(|1-z|^kN^{k+1-\delta}\right) K(z)dz\\
    &\ll_k \int_{\mathfrak{M}} \left(|1-z|^{k-1}N^{k+1-\delta}\right)dz\\
    &\ll_k \left(N^{\frac{\delta}{k+1}-1}\right)^{k-1}N^{k+1-\delta}N^{\frac{\delta}{k+1}-1}\\
    &= N^{1-\frac{\delta}{k+1}}.
\end{align*}

On $\mathcal{M}$, by Cauchy-Schwarz inequality, we have
\begin{align*}
     \int_{\mathcal{M}}\left(F(z)-\dfrac{1}{1-z}\right)K(z)dz&\ll  \left(\int_{\mathcal{M}} |K(z)|^2 dz \right)^{1/2} \left( \int_{\mathcal{M}} \left|F(z)-\dfrac{1}{1-z}\right|^2dz\right)^{1/2}
\end{align*}
Moreover, similar to \cite{Bhowmik-Ruzsa} for $k=2$, we have the estimations 
\begin{align*}
     &\int_{\mathcal{M}}\left|F(z)-\dfrac{1}{1-z}\right|^2dz \leq 
     \int_{|z|=R}\left|F(z)-\dfrac{1}{1-z}\right|^2dz \ll
     N\log N\\
     &\int_{\mathcal{M}} |K(z)|^2 dz \ll \int_{\mathcal{M}} \dfrac{1}{|1-z|^2} dz \ll N^{1-\frac{\delta}{k+1}}.
\end{align*}
Hence, we have an estimate on the minor arc
\begin{align}
    \int_{\mathcal{M}}\left(F(z)-\dfrac{1}{1-z}\right)K(z)dz\ll N^{1-\frac{\delta}{2(k+1)}} (\log N)^{1/2}.
\end{align}
Combining the major arc, minor arc and from \eqref{psi}, we obtain 
\begin{align*}
    \psi(N)-N \ll_k N^{1-\frac{\delta}{2(k+1)}} (\log N)^{1/2}. 
\end{align*}
Therefore, for any non-trivial zero $\rho$ of Riemann zeta function, we have 
$$\Re(\rho)<1-\dfrac{\delta}{2(k+1)}<1.$$

\section{Proof of Theorem \ref{2}}
Assuming the Riemann Hypothesis, we can easily deduce the asymptotic formula of $S_k(X)$ in Theorem \ref{2}. Now we prove the reverse in the following steps.

\textit{Step 1: } Granville  showed a formula of $S_k(X)$ without using the RH ((1.3), \cite{granville1,granville2}), that is 
\begin{align}\label{S-B}
    S_k(X)=\dfrac{X^k}{k!}+\sum_\rho r_k(\rho)\dfrac{X^{\rho+k-1}}{\rho+k-1}+\mathcal{O}_k(X^{k-2+\frac{4B+2}{3}+o(1)}),
\end{align}
where $B=\sup \{\Re(\rho): \zeta(\rho)=0\}$  and  
$$r_k(\rho):=-\dfrac{k}{\rho\dots(\rho+k-2)}.$$
\begin{rem}
    We know that $1/2 \leq B \leq 1$.
\end{rem}
\textit{Step 2:} We define the corresponding Dirichlet series of $G_k(n)$
\begin{align*}
    f_k(s)=\sum_{n\geq 1}\dfrac{G_k(n)}{n^s}.
\end{align*}
We note that this series converges absolutely  in  $\{\Re(s)>k\}$.
From \eqref{S-B} in step 1, we have
\begin{align*}
    f_k(s)&=s\int_1^\infty S_k(x)x^{-s-1}dx\\
    &=s\int_1^\infty \left(\dfrac{x^k}{k!}+\sum_\rho r_k(\rho)\dfrac{x^{\rho+k-1}}{\rho+k-1}+\mathcal{O}_k(x^{k-2+\frac{4B+2}{3}+o(1)})\right) x^{-s-1}dx\\
    &=\dfrac{1}{(s-k)(k-1)!}+\sum_\rho\dfrac{r_k(\rho)}{s-\rho-k+1}+s\int_1^\infty \mathcal{O}_k(x^{k-2+\frac{4B+2}{3}+o(1)}) x^{-s-1} dx\\
    &+\dfrac{1}{k!}+\sum_\rho\dfrac{r_k(\rho)}{\rho+k-1}.
\end{align*}
From the above, the series $f_k(s)$  is analytic on  $\{\Re(s)>k\}$ and can be continued meromorphically to the haft plane $\{\Re(s)>k-2+ \frac{4B+2}{3}\}$.

\textit{Step 3:} We assume that $B<1$. Then we obtain 
\begin{align}\label{infB}
    k-1+B= \inf \{\sigma_0 \geq k-\dfrac{1}{2}: f_k(s)-\dfrac{1}{(s-k)(k-1)!} \text{ is analytic on }\Re(s) > \sigma_0\}.
\end{align}
We then prove \eqref{infB}. From step 2, the right-hand side of \eqref{infB} is at most $k-2+ \frac{4B+2}{3}\leq k-1+B$, since $B\leq 1$. 

For the reverse inequality, the right-hand side of \eqref{infB} is at least $k-\dfrac{1}{2}$, then \eqref{infB} is true if $B=\dfrac{1}{2}$. So we can assume $\dfrac{1}{2}<B<1$. Hence
\begin{align*}
    \max \{k-2+ \frac{4B+2}{3}, k-\dfrac{1}{2}\}<k-1+B.
\end{align*}
There exists $\epsilon>0$ such that \begin{align*}
    \max \{k-2+ \frac{4B+2}{3}, k-\dfrac{1}{2}\}<k-1+B-\epsilon.
\end{align*}
By the definition of $B$, there exists a non-trivial zero $\rho$ such that
\begin{align*}
    B-\epsilon<\Re(\rho).
\end{align*}
Thus we obtain
\begin{align*}
    k-\dfrac{1}{2}<k-1+B-\epsilon<\Re(\rho+k-1).
\end{align*}
We consider in the  haft plane $\{ \Re(s) >k-1+B-\epsilon\}$, $f_k(s)$ has a pole at $\rho+k-1$. So the right-hand side of \eqref{infB} $\geq k-1+B-\epsilon$. Let $\epsilon$ tend to $0$, the proof of $\eqref{infB}$ is completed.

\textit{Step 4:} Let $$E_k(X)=S_k(X)-\dfrac{X^k}{k!}.$$
Then by the assumption of this part, $E_k(X)\ll_k X^{k-\frac{1}{2}+\epsilon}$.

Moreover, 
\begin{align*}
    s\int_1^\infty E_k(x)x^{-s-1}dx= f_k(s)+\dfrac{s}{(s-k)k!}.
\end{align*}
Then 
\begin{align}\label{step4}
     f_k(s)-\dfrac{1}{(s-k)(k-1)!}=s\int_1^\infty E_k(x)x^{-s-1}dx+\dfrac{1}{k!}.
\end{align}
We note that the right-hand side of \eqref{step4} is analytic on
$\{ \Re(s) >k-\dfrac{1}{2}\}$. Then from Step 3, we obtain
$$k-1+B\leq k-\dfrac{1}{2},$$
$$B \leq \dfrac{1}{2}.$$
We conclude $B=\dfrac{1}{2}$, that means the Riemann Hypothesis is true. 

\textit{Step 5: } We prove that $B=1$ is impossible. 

In fact, by assumption,
$$S_k(X)=\dfrac{X^k}{k!}+\mathcal{O}_k(X^{k-\delta}),$$ with $\delta=\frac{1}{2}-\epsilon$. Applying Theorem \ref{3}, we obtain $B<1$. 
\section{Proof of the omega-result Theorem \ref{omega} }
To prove Theorem \ref{omega}, we use the idea in \cite{BP} that if an integer $n$ is divisible by many small primes, then $G_k(n)$ should be large. We just need to prove
$$G_k(n)=\Omega(n^{k-1}\log \log n).$$

Let  $q=\prod_{\substack{p<x\\p\nmid q_1}}p$ be the product of primes which are less than $x$ and not divisible by $q_1$, where $q_1$ is the exceptional modulus up to $q$ if there exists a Siegel's zero. From Lemma 4 \cite{BP}, for $(a,q)=1$ we have
\begin{align*}
    \psi(2x, q, a)=\sum_{\substack{n\leq 2x\\n=a (q)}}\Lambda(n) \geq \dfrac{x}{2\phi(q)}.
\end{align*}
Let $b$ be an integer coprime to $q$, then 
\begin{align*}
    \sum_{\substack{n\leq 4x\\ n=b(q)}}G_2(n)&\geq \sum_{(a,q)=1} \psi (2x,q,a)\psi(2x,q,b-a)\geq \dfrac{x^2}{4\phi(q)},\\
    \sum_{\substack{n\leq 6x\\ n=b(q)}}G_3(n)&\geq \sum_{(a,q)=1} \psi (2x,q,a)\sum_{\substack{n\leq 4x\\ n=b-a(q)}}G_2(n)\geq  \dfrac{x^3}{8\phi(q)},\dots,\\
    \sum_{\substack{n\leq 2(k-1)x\\ n=b(q)}}G_{k-1}(n)&\geq \sum_{(a,q)=1} \psi (2x,q,a)\sum_{\substack{n\leq 2(k-2)x\\ n=b-a(q)}}G_{k-2}(n)\geq  \dfrac{x^{k-1}}{2^{k-1}\phi(q)}.
\end{align*}
Thus,
\begin{align*}
    \sum_{\substack{n\leq 2kx\\q|n}}G_k(n)\geq \sum_{(a,q)=1} \psi (2x,q,a)\sum_{\substack{n\leq 2(k-1)x\\ n=q-a(q)}}G_{k-1}(n)\geq  \dfrac{x^{k}}{2^{k}\phi(q)}.
\end{align*}
 So that
\begin{align*}
    \dfrac{2kx}{q}\max_{n\leq 2kx}G_k(n) &\geq  \dfrac{x^{k}}{2^{k}\phi(q)}.
\end{align*}
Therefore, we obtain
\begin{align*}
    \max_{n\leq 2kx}G_k(n) &\geq \dfrac{x^{k-1}}{2^{k+1}}\dfrac{q}{\phi(q)}=\dfrac{x^{k-1}}{2^{k+1}}\prod_{p<x}\left(1-p^{-1}\right)^{-1}\prod_{p_1 | q_1}\left(1-p_1^{-1}\right)\\
    &\gg x^{k-1}\log \log x.
\end{align*}
\section*{Acknowledgement} 
I would like to thank Gautami Bhowmik for her guidance together with many helpful  comments during the preparation of this paper.

 \bibliographystyle{amsplain}
\bibliography{anbib}

\end{document}